\documentclass[10pt]{amsart}
    \usepackage{amsmath}
    \usepackage{amssymb}
    \usepackage[all,arc]{xy}
    \usepackage{enumerate}
    \usepackage{mathrsfs}
    \usepackage{amsthm}
    \usepackage{bm} 
    \usepackage{color}
    \usepackage{graphicx}
    \usepackage{graphics}
    \usepackage{enumitem}
    \usepackage{extarrows}
    \usepackage{verbatim}
    \usepackage{bbold}
    \usepackage{dsfont}
    \usepackage{amsmath,tikz-cd}
    \usepackage{tensor} 
    
\usepackage[style=alphabetic, sorting=nyt]{biblatex}
    \makeatletter
    \newcommand{\svdots}{%
    \vbox{%
    \scriptsize
    \baselineskip 2\p@
    \lineskiplimit \z@
    \hbox {.}\hbox {.}\hbox {.}%
  }%
}
\makeatother
    \newtheorem{thm}{Theorem}[section]
    \newtheorem{cor}[thm]{Corollary}
    \newtheorem{prop}[thm]{Proposition}
    \newtheorem{lem}[thm]{Lemma}

    \theoremstyle{definition}
    \newtheorem{defn}[thm]{Definition}

    \theoremstyle{remark}
    
    \newtheorem{rem}[thm]{Remark}

    \newcommand{\Z}{\mathds{Z}}

    \newcommand{\cO}{\mathcal{O}}

    \newcommand{\modd}{\mathrm{mod{-}}}
    
    \newcommand{\coh}{\mathrm{coh}}

    \newcommand{\Proj}{\mathrm{Proj}}

    \newcommand{\Hom}{\mathrm{Hom}}

    \makeatletter
    \let\c@equation\c@thm
    \makeatother
    \numberwithin{equation}{section}

\begin{document}
    
\title{Blow ups of $\mathds{P}^n$ as quiver moduli for exceptional collections}

\author{Xuqiang QIN}
\thanks{Indiana University}
\address{Department of Mathematics, Indiana University, 831 E. Third St., Bloomington, IN 47405, USA}
\email{qinx@iu.edu}

\begin{abstract}
Suppose $P^n_m$ is the blow up of $\mathds{P}^n$ at a linear subspace of dimension $m$, $\mathcal{L}=\{L_1,\ldots,L_r\}$ is a (not necessarily full) strong exceptional collection of line bundles on $P^n_m$. Let $Q$ be the quiver associated to this collection. One might wonder when is $P^n_m$ the moduli space of representations of $Q$ with dimension vector $(1,\ldots,1)$ for a suitably chosen stability condition $\theta$: $S\cong M_\theta$. In this paper, we achieve such isomorphism using $\mathcal{L}$ of length $3$.
As a result, $P^n_m$ is the moduli space of representations of a very simple quiver. Moreover, we realize the blow up as morphism of moduli spaces for the same quiver.
\end{abstract}
    
    \maketitle
    
    \section{Introduction}
    Let $X$ be a smooth projective variety over an algebraically closed field $\mathbf{k}$ of characteristic $0$. Recall that objects $\mathcal{E}_1,\ldots,\mathcal{E}_n$ in the bounded derived category of coherent sheaves $\mathcal{D}^b(\coh(X))$ forms a exceptional collection if 
    \begin{enumerate}
          \item $\Hom(\mathcal{E}_i,\mathcal{E}_i[m])=\mathbf{k}$ if $m=0$ and is $0$ otherwise;
          \item $\Hom(\mathcal{E}_i,\mathcal{E}_j[m])=0$ for all $m\in\Z$ if $j<i$
    \end{enumerate}
    
    An exceptional collection is strong if in addition:
    $\Hom(\mathcal{E}_i,\mathcal{E}_j[m])=0$ for all $i,j$ if $m\neq 0$. It is full if the smallest triangulated subcategory of $\mathcal{D}^b(\coh(X))$ containing $\mathcal{E}_1,\ldots,\mathcal{E}_n$ is itself.\
    
   In this paper we are only concerned with the case when the objects $\mathcal{E}_i$ are line bundles and the exceptional collection is strong.
   In this situation, we can consider the finite dimensional associative algebra
   \begin{equation*}
          \mathcal{A}=\mathrm{End}(\oplus_{i=1}^n \mathcal{E}_i)
    \end{equation*}
    It is well known that if the collection is full, there is an exact equivalence of derived categories
    \begin{align*}
        \mathbf{R}\mathrm{Hom}(\oplus_{i=1}^n \mathcal{E}_i,-):\mathcal{D}^b(\coh(X))\to D^b(\modd \mathcal{A})
    \end{align*}
    whose inverse is given by 
    \begin{align*}
        (-)\otimes^L(\oplus_{i=1}^n \mathcal{E}_i):D^b(\modd \mathcal{A})\to {D}^b(\coh(X))
    \end{align*}
    This gives a non-commutative interpretation of the derived category of $X$. We note that when we input the structure sheaf $\cO_x$ of a close point $x\in X$ into the first functor, we obtain:
    \begin{align*}
        \mathbf{R}\mathrm{Hom}(\oplus_{i=1}^n \mathcal{E}_i,\cO_x)&=\mathrm{Hom}(\oplus_{i=1}^n \mathcal{E}_i,\cO_x)\\
        &=\oplus_{i=1}^n (\mathcal{E}_i^{\vee})_x
    \end{align*}
    Note since $\mathcal{A}^{op}=\mathrm{End}\big(\oplus_{i=1}^n (\mathcal{E}_i^{\vee})\big)$ is a finite dimensional algebra, there exist a bound quiver $(Q,I)$ such that giving an $\mathcal{A}^{op}$-module is equivalent to giving a representation of $(Q,I)$.\\ 
    
    King\cite{King} proved that when restricted by a stability condition $\theta$, the moduli space $M_{\theta}(\alpha)$ of semistable representations of a bound quiver $(Q,I)$ with any dimension vector $\alpha$ is a projective scheme. Moreover, if $I=0$, i.e. the quiver has no relations, the moduli space $M^S_{\theta}(\alpha)$ of stable representation is a smooth projective variety. In our situation, for each point $x\in X$, one can associate the representation $\oplus_{i=1}^n (\mathcal{E}_i^{\vee})_x$ of $\mathcal{A}^{op}$, which has dimension vector $(1,\ldots,1)$. This provides a tautological map \begin{align*}
        T_0:X\to \mathcal{R}ep_{(1,\ldots,1)}(Q)
    \end{align*}  
    to the moduli stack of representation of $(Q,I)$ with dimension vector $(1,\ldots,1)$. Following \cite{BP}, we note $T_0$
    induce a tautological rational map $T:X\dashrightarrow M_\theta$  where $M_\theta$ is the moduli space of semistable representation of $(Q,I)$ with dimension vector $(1,\ldots,1)$. We note $M_\theta$ is a projective scheme and 
    \begin{align*}
        T(x)&=\mathrm{Hom}(\oplus_{i=1}^n \mathcal{E}_i,\cO_x)\\
        &=\oplus_{i=1}^n (\mathcal{E}_i^{\vee})_x
    \end{align*}
    if $x\in X$ is in the domain of $T$. It is also important to notice that $T$ can be similarly defined even if the collection of line bundles is not full.\\
    
    It is natural to wonder when $T$ is a morphism and the relation between $M_{\theta}$ and $X$ for various $\theta$. \cite{QZ} provided some answers for this problem for rational surfaces. On the other hand, one can also consider the $(n+1)$-Kronecker quiver, which can be thought of as the quiver associated to the (not full) strong exceptional collection $\{\cO_{\mathds{P}^n},\cO_{\mathds{P}^n}(1)\}$ on $\mathds{P}^n$. It is well know that for appropriate stability condition, $T$ is an isomorphism. In this paper, we prove a similar result on $P^n_m$:
    \begin{thm}\label{main}
        Let $P^n_m$ be the blow up of $\mathds{P}^n$ with center a linear subspace of dimension $m$. Let $E$ be the exceptional divisor and $H$ the pull back of hyperplane section. Then  $\{\mathcal{O}_{P^n_m},\mathcal{O}_{P^n_m}(H-E),\mathcal{O}_{P^n_m}(H)\}$ is a strong exceptional collection of line bundles that is not full. Let $\mathcal{A}=\mathrm{End}\Big(\mathcal{O}_{P^n_m}\oplus\mathcal{O}_{P^n_m}(H-E)\oplus\mathcal{O}_{P^n_m}(H)\Big)$ be the endomorphism algebra. Let $Q$ be the quiver for $\mathcal{A}^{op}$. Then one can choose (many) stability conditions so that the moduli space of semistable representations of $Q$ with dimension vector $(1,1,1)$ is a fine moduli space and the tautological rational map
        \begin{align*}
            T:P^n_m\dashrightarrow M_\theta
        \end{align*}
        is an isomorphism.
    \end{thm}
    
    \begin{cor}
        There are many stability conditions $\theta$ such that $P^n_m$ is the moduli space of semistable representations with dimension vector $(1,1,1)$ of the following quiver $Q$
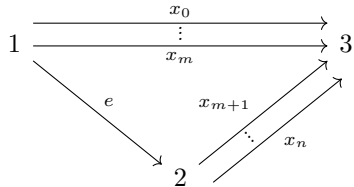
\[
    \begin{tikzcd}[column sep=huge,row sep=huge]
    1
    \arrow[rr,shift left=2ex,"x_0"]
    \arrow[rr,"\svdots","x_m" swap]
    \arrow[dr,"e"]
    && 3
    \\
    &2\arrow[ur,"x_{m+1}" pos=0.45]
     \arrow[ur,shift right=2ex,"\svdots" {sloped,pos=0.4},"x_n" swap]
    \\
    \end{tikzcd}
    \]\end{cor}
The proof of the theorem follows the idea in \cite{QZ} by thinking of $\{\mathcal{O}_{P^n_m},\mathcal{O}_{P^n_m}(H-E),\mathcal{O}_{P^n_m}(H)\}$ as an 'augmentation' of $\{\cO_{\mathds{P}^n},\cO_{\mathds{P}^n}(1)\}$ on $\mathds{P}^n$. However, many arguments were simplified because the quiver in the present case is nicer and has no relations. \\

Our second theorem shows that by varying the stability condition, we can also obtain $\mathds{P}^n$ as a moduli space of representations of $Q$, and realize the natural contraction $\pi:P^n_m\to \mathds{P}^n$ as a morphism between quiver moduli.
\begin{thm}\label{second}
    Let $\theta'=(-p,0,p)$ for $p\in\Z_{>0}$, then the moduli space $M_{\theta'}$ of $\theta'$-semistable representation of $Q$ with dimension vector $(1,1,1)$ is isomorphic to $\mathds{P}^n$ via the tautological map. Moreover, the morphism $id$ induced  by the identity map $\iota:\mathbf{k}[x_0,\ldots,x_n,e]\to \mathbf{k}[x_0,\ldots,x_n,e]$ on the coordinate ring of  quiver variety of $Q$ descents to the natural projection $\pi':M_\theta\to M_{\theta'}$
\end{thm}

If one uses the collection $\{\mathcal{O}_{P^n_m},\mathcal{O}_{P^n_m}(E),\mathcal{O}_{P^n_m}(H)\}$, which corresponds to the quiver $Q^{op}$, the same results follow from \cite{AC} or \cite{CS}. In fact, much more general results were proved in these two papers for quivers which come from a collection of globally generated line bundles using GIT techniques. We note that they used quivers corresponding to the endomorphism algebra while ours corresponds to the opposite.
In particular the present collection $\{\mathcal{O}_{P^n_m},\mathcal{O}_{P^n_m}(E),\mathcal{O}_{P^n_m}(H)\}$ satisfies the globally generating assumption in \cite{AC} or \cite{CS}, while the collection in Theorem 1.1 does not.

\subsection*{Notations and Conventions}
\begin{itemize}
    \item $P^n_m$ is the blow up of $\mathds{P}^n$ centered at a linear subspace of dimension $m$. We require $n\geq 2$ and $0\leq m\leq n-2$.
          
    \item For a sheaf $\mathcal{F}$ on a scheme $X$, we use $h^i(X,\mathcal{F})$ to denote the dimension of $H^i(X,\mathcal{F})$.
\end{itemize}

\section{Preliminaries}
\subsection{Geometry of $P^n_m$} In this section we provide basic facts on geometry of $P^n_m$. The main references for this section are \cite{Har} \cite{LYY}.\\     
Let $\pi:P^n_m\to \mathds{P}^n$ be the blow up of $\mathds{P}^n$ with center
$\mathds{L}=\{X_{m+1}=\ldots=X_n=0\}$ being a linear subspace of dimension $m$. Then
\begin{align*}
    \mathcal{N}_{\mathds{L}/\mathds{P}^n}=\mathcal{O}_{\mathds{L}}(-1)^{\oplus (n-m)}
\end{align*}
Thus the exceptional divisor $E\cong \mathds{L}\times \mathds{P}^{n-m-1}\cong \mathds{P}^m\times \mathds{P}^{n-m-1}$. The $P^n_m$ is a smooth toric variety whose Picard group is 
\begin{align*}
    Pic(X)\cong \Z[H]\oplus \Z[E]
\end{align*}
where $H$ is the pull back of a hyperplane section of $\mathds{P}^n$. The canonical divisor of $P^n_m$ is given by
\begin{align*}
    K_{P^n_m}=-(n+1)H+(n-m-1)E
\end{align*}
The Chow ring of $P^n_m$ has the following presentation
\begin{align*}
    A(P^n_m)=\Z[H,E]/\big((H-E)^{n-m},H^m-(H-E)H^{m-1}\big)
\end{align*}\\

Recall a divisor $D$ on a smooth projective variety is called strong left orthogonal if
\begin{align*}
    h^i(\cO(D))=0
\end{align*}
for all $i>0$ and 
\begin{align*}
    h^{i}(\cO(-D))=0
\end{align*}
for all $i\geq 0$.
\begin{lem}\label{slo}
    The divisors $H,E,H-E$ on $P^n_m$ are strong left orthogonal.
\end{lem}
\begin{proof}
    Consider the short exact sequence
    \begin{align*}
        0\to\cO_{P^n_m}(-E)\to \cO_{P^n_m}\to \cO_E\to 0
    \end{align*}
    Taking cohomology we obtain long exact sequence:
    \begin{align*}
        \cdots\to H^i(\cO_{P^n_m}(-E))\to H^i(\cO_{P^n_m})\to H^i(\cO_E)\to \cdots
    \end{align*}
    Since $E\cong \mathds{P}^m\times\mathds{P}^{n-m-1}$, by Kunneth formula, we get $H^0(\cO_E)=k$ and $H^i(\cO_E)=0$ for $i>0$. Moreover, it is clear that the map $H^0(\cO_{P^n_m})\to H^0(\cO_E)$ is an isomorphism, hence $h^i(\cO_{P^n_m}(-E))=0$ for all $i>0$. Since $E$ is effective, $h^0(\cO_{P^n_m}(-E))=0$. \\
    Using the same argument, we can show $h^i(\cO_{P^n_m}(-H))=h^i(\cO_{P^n_m}(-H+E))=0$ for all $i\geq 0$.\\
    
    Twist the short exact sequence \begin{align}\label{SE}
        0\to\cO_{P^n_m}(-E)\to \cO_{P^n_m}\to \cO_E\to 0
    \end{align}
    by $\cO_{P^n_m}(E)$, we obtain
    \begin{align*}
        0\to\cO_{P^n_m}\to \cO_{P^n_m}(E)\to \cO_E(E)\to 0
    \end{align*}
    Taking cohomology we obtain long exact sequence:
    \begin{align*}
        \cdots\to H^i(\cO_{P^n_m})\to H^i(\cO_{P^n_m}(E))\to H^i(\cO_E(E))\to \cdots
    \end{align*}
    Denoting $Pic(E)\cong Pic(\mathds{L})\times Pic(\mathds{P}^{n-m-1})$. Then $\cO_E(E)=\cO_E(1,-1)$. Using Kunneth formula, we see $h^i(\cO_E(E))=0$ for all $i\geq 0$. From this one easily see $h^i(\cO_{P^n_m}(E))=0$ for all $i> 0$.\\
    Similarly, we can get $h^i(\cO_{P^n_m}(H))=0$ for all $i>0$.
    
    Twist (\ref{SE}) by $O(H)$  we obtain short exact sequence 
    \begin{align*}
         0\to\cO_{P^n_m}(H-E)\to \cO_{P^n_m}(H)\to \cO_E(H)\to 0
    \end{align*}
    Note $\cO_E(H)=\cO_E(1,0)$. Again using Künneth formula and noting $H^0(\cO_{P^n_m}(H))\to H^0(\cO_E(H))$ is surjective, we see $h^i(\cO_{P^n_m}(H-E))=h^i(\cO_{P^n_m}(H))=0$ for all $i>0$. This finishes the proof.
\end{proof}

\subsection{Quivers and quiver representations} See also \cite{Br}.\\
A quiver $Q$ is given by two sets $Q_{vx}$ and $Q_{ar}$, where the first set is the set of vertices and the second is the set of arrows, along with two functions $s,t:Q_{ar}\to Q_{vx}$ specifying the source and target of an arrow. The path algebra $\mathbf{k}Q$ is the associative $\mathbf{k}$-algebra whose underlying vector space has a basis consists of elements of $Q_{ar}$. The product of two basis elements is defined by concatenation of paths if possible, otherwise $0$. The product of two general elements is defined by extending the above linearly.\
    A bound quivers is a pair $(Q,I)$. Here $Q$ is a quiver and $I$ is a two sided ideal of $\mathbf{k}Q$ generated by elements of the form $\sum_{i=1}^nk_ip_i$, where $k_i\in \mathbf{k}^*$ and $p_i$ are paths with same heads and same tails for $i\in\{1,\ldots,n\}$. We simply use $Q$ to denote this pair when the existence of $I$ is understood.\\ 
    
    Let $Q$ be a quiver. A quiver representation $R=(R_v,r_a)$ consists of a vector space $R_v$ for each $v\in Q_{vx}$ and a morphism of vector spaces $r_a:R_{s(a)}\to R_{t(a)}$ for each $a\in Q_{ar}$. For a bound quivers $(Q,I)$, a representation  $R=(R_v,r_a)$ is same as above, with the additional condition that 
    \begin{align*}
        \sum_{i=1}^nk_ir_{p_i}=0
    \end{align*} if $\sum_{i=1}^nk_ip_i$ is a generator of $I$. A subrepresentation of $R$ is a pair $R'=(R_v',r_a')$ where $R_v'$ is a subspace of $R_v$ for each $v\in Q_{vx}$ and $r'_a:R'_{s(a)}\to R'_{t(a)}$ a morphism of vector spaces for each $a\in Q_{ar}$ such that $$r'_a=r_a|_{R'_{s(a)}}$$ and 
    \begin{equation}\label{subset}
        r_a(R'_{s(a)})\subset R'_{t(a)}
    \end{equation}
    Thus we have the commutative diagram
    \[ \begin{tikzcd}
      R'_i \arrow[r,"r'_a"] \arrow{d}{\iota_i} & R'_j \arrow{d}{\iota_j} \\%
      R_i \arrow{r}{r_a}& R_j
      \end{tikzcd}
      \]
    for any arrow $a$ from $i$ to $j$.
    We use $R'\subset R$ to denote that $R'$ is a subrepresentation of $R$.

    If the vertices of a quiver has a natural ordering, as it will be the case when we discuss quiver of sections of an exceptional collection of line bundles, we define dimension vector $\vec{d}$ so $d_i$ is the dimension of the vector space $R_i$ at that vertex. We call the set of vertices where $R_v$ has positive dimension the support of $R$.\\
    
    In this paper, we are particularly interested in representations with dimension vector $\mathbf{1}=(1,\ldots,1)$. Notice when $R$ is a representation with dimension vector $\mathbf{1}$, and $R'\subset R$, all the inclusion maps $$\iota_k:R'_k\to R_k$$ are either zero map or identity. We prove the following easy lemma:
    \begin{lem}\label{subrep}
      Let $(Q,I)$ be a bound quivers whose vertices are label by $\{0,1,2,\ldots,n\}$ and $R$ be a representation of $Q$ with dimension vector $\mathbf{1}$.Then any subrepresentation $R'$ is determined by its dimension vector $\vec{d}$. Moreover, a vector $\vec{d}$ of size $n+1$ with entries $0$ and $1$ is the dimension vector of a subrepresentation of $R$ if and only if $r_a=0$ for all $a\in Q_{ar}$ with $d_{s(a)}=1$ and $d_{t(a)}=0$.
      \end{lem}
      \begin{proof}
      Since $\dim R_i=1$, its subspaces are determined by dimensions. Moreover, we see the morphism of subspaces $r'_a$ are restrictions of $r_a$, hence the dimension vector $\vec{d}$ determines $R'_i$ for all $i\in Q_{vx}$.\\
      Given any vector $\vec{d}$ as in the second part of the lemma, it is the dimension of a vector subspace if (\ref{subset}) is satisfied. Note (\ref{subset}) is always true unless for arrows with $d_{s(a)}=1$ and $d_{t(a)}=0$, in which case we must have $r_a=0$.
      \end{proof}
    
    \subsection{Moduli space of semistable representations of a quiver} See also \cite{King},\cite{Re}.\\
    Given a bound quivers $(Q,I)$,a weight is an element $\theta\in\Z^N$ where $N=|Q_{vx}|$ such that $\sum_{i=1}^{N}\theta_i=0$. Let $\theta=(\theta_1,\ldots,\theta_{N})$ be a weight, we defined its toric form to be
    \begin{align*}
        (-\theta_1,-\theta_1-\theta_2,\ldots,-\theta_1-\theta_2-\ldots-\theta_{N-1})\in \Z^{n-1}
    \end{align*}
    It is an easy exercise to see that one can recover a weight from its toric form.
    \begin{defn}
      A weight is admissible if every entry of its toric form is a positive integer.
    \end{defn}
    For a weight $\theta$, the weight function is defined by by:
    \begin{equation*}
        \theta(S)=\sum_{i=1}^N d_i \theta_i
    \end{equation*}
    where $S$ is a representation of $Q$ and $d_i$ and $\theta_i$ are the $i$-th entries of $\vec{d}$ and $\theta$ respectively. We recall the definition of semi-stability:
    \begin{defn}
        A representation $R$ is $\theta$-semistabe if for any subrepresentation $R'\subset R$
    \begin{equation*}
        \theta(R')\geq0
    \end{equation*}
    $R$ is $\theta$-stable if all the above inequalities are strict.
    \end{defn}
    
    We restrict our attention to $R$ with dimension vector $\mathbf{1}$. Given a bound quivers $(Q,I)$, we can associate to it an affine shceme $\mathrm{Rep(Q)}$ called the representation scheme of $(Q,I)$. The coordinate ring of this affine shceme is the quotient of $\mathbf{k}[a\in Q_{ar}]$ by the ideal $J$ which is generated by generators $\sum_{i=1}^nk_ip_i$ of $I$ treated as  elements in the above polynomial ring. It is obvious from the definition that closed points of representation scheme are in 1-to-1 correspondence with representations of $Q$ with dimension vector $\mathbf{1}$. For a weight $\theta$, the set of $\theta$-semistable representations forms an open subscheme $\mathrm{Rep(Q)}^{SS}_\theta$ of $\mathrm{Rep(Q)}$, the set of $\theta$-stable representations forms an open subscheme $\mathrm{Rep(Q)}^{S}_\theta$ of $\mathrm{Rep(Q)}^{SS}_\theta$.
    
    The group $(\mathbf{k}^*)^{Q_{vx}}$ acts by incidence on $\mathrm{Rep(Q)}$, in other words, it acts by $(g\cdot a)=g_{t(a)}r_ag^{-1}_{s(a)}$. Apparently, the diagonal subgroup $\mathbf{k}_{\text{diag}}^*$ of $(\mathbf{k}^*)^{Q_{vx}}$ consisting of elements of the form $(k,k,\ldots,k)$ for $k\in\mathbf{k}^*$ acts trivially on $\mathrm{Rep(Q)}$. So it is natural to only consider the action of $\mathrm{PGL(\mathbf{1})}:=(\mathbf{k}^*)^{Q_{vx}}/\mathbf{k}_{\text{diag}}^*$. 
    \begin{defn}
      Two representations of dimension vector $\mathbf{1}$ are isomorphic if they are in the same orbit under the action of $\mathrm{PGL(\mathbf{1})}$.\
    \end{defn}

    Give a weight $\theta$, the moduli space of $\theta$-semistable representation with dimension vector $\mathbf{1}$ is the GIT quotient
    \begin{align*}
        M_\theta:&=\mathrm{Rep(Q)}//_\theta \mathrm{PGL(\mathbf{1})}\\
        &=\mathrm{Rep(Q)}^{SS}_\theta//\mathrm{PGL(\mathbf{1})}
    \end{align*}
    We mention a few facts about $M_\theta$. For details, the readers are referred to \cite{King}.
    An equivalent definition of $M_{\theta}$ is to consider the graded ring
       $$ B_{\theta}= \bigoplus_{r\geq 0}B(r\theta)$$
    where $B(r\theta)$ is $r\theta$-semi-invariant functions in the coordinate ring of $\mathrm{Rep(Q)}$. Then the GIT quotient is defined as 
    \begin{equation*}
        M_\theta=\Proj(B_\theta)
    \end{equation*}
    From this definition, it is easy to see that $M_{\theta}$ is a reduced projective scheme. Note if all $\theta$-semistable representations are $\theta$-stable, i.e. $\mathrm{Rep(Q)}^{SS}_{\theta}=\mathrm{Rep(Q)}^{S}_{\theta}$, then $M_{\theta}$ is the fine moduli space of $\theta$-stable representations, in particular, the closed points of $M_{\theta}$ are in 1-to-1 correspondence with the isomorphism classes of $\theta$-stable representations.
    \
    We now give an easy criterion for obtaining fine moduli spaces as above.
    \begin{lem}\label{fine}
    With the notions above, if for any proper nonempty subset $P$ of $Q_{vx}$, we have $\sum_{i\in P}\theta_i\neq 0$, then any semistable representation $R$ is in fact stable. In particular, $M_{\theta}$ is a fine moduli space. 
    \end{lem}
    \begin{proof}
      If $R$ is strictly semistable, then there exist a proper nonzero subrepresentation $R'$ such that $ \theta(R')=\sum_{i\in supp(R')}\theta_i =0$, but this cannot happen given the conditions in the statement.
    \end{proof}

    \subsection{Quivers of Sections} The main reference for this section is Craw-Smith\cite{CS} and Craw-Winn\cite{CW}. We mention that our indexing is different since we are concerned with the quiver with path algebra $\mathcal{A}^{op}$ instead of $\mathcal{A}$ as in the introduction.\\
    \noindent Let $\mathcal{L}=\{L_1,L_2,\ldots,L_n\}$ be a collection of line bundles on a projective variety $X$. For $1\leq i,j\leq n$, we call a section $s\in H^0(X,L_j^{\vee}\otimes L_i)$ irreducible if $s$ does not lie in the images of the multiplication map
    \begin{equation*}
        H^0(X,L_j^{\vee}\otimes L_k)\otimes_{\mathbf{k}}H^0(X,L_k^{\vee}\otimes L_i)\to H^0(X,L_j^{\vee}\otimes L_i)
    \end{equation*}
    for $k\neq i,j$. 
    \begin{defn}
      The quiver of sections of the collection $\mathcal{L}$ on $X$ is defined to be a quiver with vertex set $Q_{vx}=\{1,\ldots,n\}$ and where the arrows from $i$ to $j$ corresponds to a basis of irreducible sections of $H^0(X,L_{(n+1)-j}^{\vee}\otimes L_{(n+1)-i})$.
    \end{defn}
    We mention one of the basic properties of a quiver of sections.
    \begin{lem}{\cite{CW}}
    The quiver of sections $Q$ is connected, acyclic and $1\in Q_{vx}$ is the unique source.
    \end{lem}
    The quiver of sections only include information about the sections in $H^0(X,L_{(n+1)-j}^{\vee}\otimes L_{(n+1)-i})$, but left relations between them behind. We now define a two sided ideal
    \begin{defn}
    Let $I_\mathcal{L}$ be a two sided ideal in $\mathbf{k}Q$
    \begin{equation*}
        I_\mathcal{L}=\big(\sum_{k=1}^Na_kp_k|p_k \text{ are paths from }i \text{ to }j\ \text{ and }\sum_{k=1}^Na_kp_k \text{ represents } 0 \text{ in } H^0(X,L_{(n+1)-j}^{\vee}\otimes L_{(n+1)-i})\big)
    \end{equation*}
    We call the pair $(Q,I_\mathcal{L})$ the bound quiver of sections of the collection $\mathcal{L}$.
    \end{defn}
    \begin{prop}{\cite{CS}\cite{CW}}
      The quotient algebra $\mathbf{k}Q/I_\mathcal{L}$ is isomorphic to $\mathcal{A}^{op}=\mathrm{End}_{\cO_X}(\oplus_{i=1}^nL_i^\vee)$ and for $1\leq i,j\leq n$, we have $e_j(\mathbf{k}Q/I_\mathcal{L})e_i\cong H^0(X,L_{(n+1)-j}^{\vee}\otimes L_{(n+1)-i})$.
    \end{prop}

    Given any weight $\theta$ for $(Q,I_\mathcal{L})$, we can consider the moduli space of semistable representations $M_\theta$. There is a tautological rational map
    \begin{equation*}
        T:X\dashrightarrow M_\theta
    \end{equation*}
    so that if $T$ is defined at $x$, then  $$T(x)=\bigoplus_{i=0}^n(L_i^\vee)_x$$ Moreover, $T$ is defined at $x$ if $\bigoplus_{i=0}^n(L_i^\vee)_x$ can be represented by a $\theta$-semistable representation.
    
    \subsection{Kronecker Quiver} Fix $n\geq 2$. Let $Q_0$ be the $(n+1)$-Kronecker quiver:
\[
    \begin{tikzcd}[column sep=huge,row sep=huge]
    1
    \arrow[r,shift left=2ex,"x_0"]
    \arrow[r,"\svdots","x_n" swap]
    & 2
    \end{tikzcd}
    \]
For $p>0$, let $\theta_0=(-p,p)$.Let $M_{\theta_0}$ be the moduli space of semistable representations of $Q_0$ with dimension vector $(1,1)$. The following fact is well known:
\begin{prop}
    The tautological rational map:
    \begin{align*}
        T_0:\mathds{P}^n\dashrightarrow M_{\theta_0}
    \end{align*}
    is an isomorphism.
\end{prop}
\begin{proof}
    By Lemma \ref{subrep}, a representation $R$ of $Q_0$ is unstable if and only if $r_{x_i}=0$ for all $i$. Note for $[a_0:\ldots:a_n]\in \mathds{P}^n$,
    \begin{align*}
        T_0([a_0:\ldots:a_n])=[(a_0,\ldots,a_m)]
    \end{align*}
    Since at least one of $a_i$ is nonzero, $T_0$ is defined on all of $\mathds{P}^n$. Moreover, it is clear that $T_0$ is one-to-one and onto, so it is an isomorphism.
\end{proof}
\begin{rem}
    Note we can identify
    \begin{align*}
        B(k\theta_0)=\mathbf{k}[x_0^{d_0}\ldots x_n^{d_n},d_0+\ldots d_n=kp]
    \end{align*}
    with a subspace vector space of the graded algebra $k[x_0,\ldots,x_n]$. Then 
    \begin{align*}
        \bigoplus_{k\geq0}B(k\theta)
    \end{align*}
    is the algebra that corresponds to the $p$-uple embedding of $\mathds{P}^n$.
\end{rem}
    
    \section{Proof of Theorem \ref{main}}
    We consider the following collection of line bundles on $P^n_m$ of length $3$
    \begin{align*}
        \{\cO_{P^n_m},\cO_{P^n_m}(H-E),\cO_{P^n_m}(H)\}
    \end{align*}
    Using Lemma \ref{slo}, we see this is a strong exceptional collection of line bundles. It is clearly not full due to its small length. Its quiver of sections $Q$ has no relations and is given by
    \[
    \begin{tikzcd}[column sep=huge,row sep=huge]
    1
    \arrow[rr,shift left=2ex,"x_0"]
    \arrow[rr,"\svdots","x_m" swap]
    \arrow[dr,"e"]
    && 3
    \\
    &2\arrow[ur,"x_{m+1}" pos=0.45]
     \arrow[ur,shift right=2ex,"\svdots" {sloped,pos=0.4},"x_n" swap]
    \\
    \end{tikzcd}
    \]
    Let $p,q\in \Z_{>0}$ and $p<q$. We define $\theta=(-p,p-q,q)$ and $\theta_0=(-p,p)$. Let $M_\theta$ be the moduli space of semistable representation of $Q$ with dimension vector $(1,1,1)$.
    \begin{lem}
        $M_\theta$ is in fact the fine moduli space of stable representation of $Q$ with dimension vector $(1,1,1)$. It is a smooth projective variety.
    \end{lem}
    \begin{proof}
        By Lemma \ref{fine}, $M_\theta$ is the same as the moduli space of stable representations of $Q$ with dimension vector $(1,1,1)$.\\
        The second part of the lemma follows from \cite{King}
    \end{proof}
    
    \begin{thm}\label{map}
      There is a natural surjective morphism
      \begin{equation*}
          F:\mathrm{Rep(Q)}\to \mathrm{Rep(Q_0)}
      \end{equation*}
    \end{thm}
    \begin{proof}
      We define a $\mathbf{k}$-algebra homomorphism
      $\phi:\mathbf{k}[X_0,\ldots,X_n]\to \mathbf{k}[x_0,\ldots,x_n,e]$ as
      follows:
      \begin{align*}
          \phi(X_i)=
          \begin{cases}
            x_i &\text{if $i\leq m$}\\
            ex_i &\text{if $i>m$}
          \end{cases}
      \end{align*}
      We let $F$ be the corresponding morphism between affine schemes.\\
      
      To see $F$ is surjective, we note for $(a_0,\ldots, a_n)\in Rep(Q)$, we have 
      \begin{align*}
          F(a_0,\ldots,a_n,1)=(a_0,\ldots, a_n)
      \end{align*}
    \end{proof}
    The next proposition shows $F$ respects the $\mathrm{PGL(\mathbf{1})}$- action.
    
    \begin{prop}\label{descent}
      Let $R_1$,$R_2$ be two representations of $Q$ with dimension vector $\mathbf{1}$. Suppose $R_1\sim R_2$, via the element $(g_1,g_2,g_3)$, then $F(R_1)\sim F(R_2)$.
    \end{prop}
    \begin{proof}
      From the construction of $F$, one directly check the element
      \begin{equation*}
          (g_1,g_3)
      \end{equation*}
      provides the equivalence.
      \end{proof}

    Let $U$ consists of representations of $Q$ so that $r_e\neq 0$.
    \begin{lem}\label{stable1}
    Suppose $R\in U$, then $R$ is $\theta$-stable if and only if
    \begin{align*}
              r_{x_i}\neq 0
    \end{align*}
    for at least one $m+1\leq i\leq n$
    \end{lem}
    \begin{proof}
        Suppose
        \begin{align*}
              r_{x_i}= 0
        \end{align*}
    for all $m+1\leq i\leq n$, there exist a subrepresentation $S\subset R$ with dimension vector $(0,1,0)$. Then
    \begin{align*}
        \theta(S)=p-q<0
    \end{align*}
    For the other direction, if
    \begin{align*}
              r_{x_i}\neq 0
    \end{align*}
    for at least one $m+1\leq i\leq n$, a nontrivial proper subrepresentation of $S$ can only have one of the following dimension vectors
          \begin{itemize}
              \item $(0,1,1)$
              \item $(0,0,1)$
          \end{itemize}
          one easily check then $R$ is stable.
    \end{proof}
    
    \begin{prop}\label{injective}
      Suppose $R_1,R_2\in U$, and $F(R_1)\sim F(R_2)$ under the action of $(g_1,g_3)$, then $R_1\sim R_2$.
    \end{prop}
    \begin{proof}
      Let $e_i$ denote the value of $e$ in $R_i$ for $i=1,2$, then $e_1e_2\neq 0$. Again by the construction of $F$, one directly checks that 
      \begin{equation*}
          \Big(g_1e_2,,g_3e_1,g_3e_2\Big)
      \end{equation*}
      provides the equivalence.
    \end{proof}
    
    \begin{lem}\label{stable2}
          Suppose $R\in \mathbf{V}(e)$ , then $R$ is $\theta$-stable if and only if 
          \begin{align*}
              r_{x_i}\neq 0
          \end{align*}
          for at least one $0\leq i\leq m$ and 
          at least one $m+1\leq i\leq n$.
    \end{lem}
    \begin{proof}
        Suppose $r_{x_i}=0$ for all $0\leq i\leq m$, there exist a subrepresentation $S\subset R$ with dimension vector $(1,0,0)$, then 
        \begin{align*}
            \theta(S)=-p<0
        \end{align*}
        Suppose $r_{x_i}=0$ for all $m+1\leq i\leq n$, there exist a subrepresentation $S\subset R$ with dimension vector $(0,1,0)$, then 
        \begin{align*}
            \theta(S)=p-q<0
        \end{align*}
        For the other direction, suppose \begin{align*}
              r_{x_i}\neq 0
          \end{align*}
          for at least one $0\leq i\leq m$ and 
          at least one $m+1\leq i\leq n$, a nontrivial proper subrepresentation of $S$ can only have one of the following dimension vectors
          \begin{itemize}
              \item $(0,1,1)$
              \item $(1,0,1)$
              \item $(0,0,1)$
          \end{itemize}
          one easily check then $R$ is stable.
    \end{proof}
    
    \begin{cor}
        The tautological rational map $T$ is a morphism.
    \end{cor}
    \begin{proof}
        Let $x\in P^n_m$. If $x\in E\cong\mathds{P}^m\times\mathds{P}^{n-m-1}$, suppose $x=([a_0,\ldots,a_m],[b_{m+1},\ldots,b_n])$, then $T(x)=[(a_0,\ldots,a_m,b_{m+1},\ldots,b_n,0)]$. Since at least one of $a_i$ and at least one of $b_i$ is nonzero, by Lemma, $T(x)$ is stable.\\
        If $x\in P^n_m\backslash E$, then we can write $x=\pi^{-1}([a_0,\ldots,a_n])$, where $a_i\neq 0$ for at least one $m+1\leq i\leq n$ since $x\notin E$. Then $T(x)=[(a_0,\ldots,a_n,1)]$. By Lemma, $T(x)$ is stable.
    \end{proof}
    
    \begin{prop}\label{commute}
        The natural morphism
      \begin{equation*}
           F:\mathrm{Rep(Q)}\to \mathrm{Rep(Q_0)}
      \end{equation*}
      descends to a projective morphism
      \begin{equation*}
          f:M_{\theta}\to M_{\theta_0}
      \end{equation*}
      which fits into a commutative diagram
      \[ \begin{tikzcd}
      P^n_m \arrow[r,"T"] \arrow{d}{\pi} & M_\theta \arrow{d}{f} \\%
      \mathds{P}^n \arrow{r}{T_0}& M_{\theta_0}
      \end{tikzcd}
      \]
    \end{prop}
    
    \begin{proof}
        By Lemma \ref{stable1} \ref{stable2}, $F$ descents to 
        \begin{align*}
            f:M_{\theta}\to M_{\theta_0}
        \end{align*}
        $f$ is projective since it is a morphism between projective schemes.\\
        
        Let $x\in P^n_m$. If $x\in E\cong\mathds{P}^m\times\mathds{P}^{n-m-1}$, suppose $x=([a_0,\ldots,a_m],[b_{m+1},\ldots,b_n])$, then $T(x)=[(a_0,\ldots,a_m,b_{m+1},\ldots,b_n,0)]$. Hence
        \begin{align*}
            F\circ T(x)=[(a_0,\ldots,a_m,0,\ldots,0)]
        \end{align*}
        Now $\pi(x)=[a_0,\ldots,a_m,0\ldots,0]$ \begin{align*}
            T_0\circ\pi(x)&=[(a_0,\ldots,a_m,0\ldots,0)]\\
            &=F\circ T(x)
        \end{align*}
        If $x\in P^n_m\backslash E$, then we can write $x=\pi^{-1}([a_0,\ldots,a_n]$. Then $T(x)=[(a_0,\ldots,a_n,1)]$ and
        \begin{align*}
            F\circ T(x)&=[(a_0,\ldots,a_n)]\\
            &=T_0\circ \pi(x)
        \end{align*}
    \end{proof}
    Let $C$ denote the closed subscheme of  $M_\theta$ containing stable orbits of representations with $r_e=0$, i.e. 
    \begin{align*}
        C=\mathbf{V}(e)^{S}//\mathrm{PGL(\mathbf{1})}
    \end{align*}
    where $\mathbf{V}(e)^{S}$ is the open subscheme of $\mathbf{V}(e)$ consisting of stable representations.
    \begin{prop}\label{fiber}
        We have
        $C\cong \mathds{P}^m\times\mathds{P}^{n-m-1}\cong\mathds{L}\times\mathds{P}^{n-m-1}$ and $f|C$ is the projection to $\mathds{L}$
    \end{prop}
    \begin{proof}
        By Lemma \ref{stable2},
        \begin{align*}
            \mathbf{V}(e)^S=(\mathds{A}^{m+1}\backslash 0)\times (\mathds{A}^{n-m}\backslash 0)
        \end{align*}
        For $(1,g_2,g_3)\in \mathrm{PGL(\mathbf{1})}$, it acts on $\mathbf{V}(e)^S$ by letting $g_2$ acts on the first component via scalar multiplication and $g_3$ acts on the second in the same way. Thus
        \begin{align*}
            C&=\mathbf{V}(e)^S//\mathrm{PGL(\mathbf{1})}\\
            &=(\mathds{A}^{m+1}\backslash 0)//k^*\times (\mathds{A}^{n-m}\backslash 0)//k^*\\            &=\mathds{P}^m\times\mathds{P}^{n-m-1}
        \end{align*}
        The fact that $f|C$ is the projection to the first component follows directly from the definition of $F$.
    \end{proof}
    
    \begin{cor}\label{open}
         $f$ induces an isomorphism between $M_\theta\backslash C$ and $M_{\theta_0}\backslash T_0(\mathds{L})$, thus $f$ is a birational morphism.
    \end{cor}
    \begin{proof}
        By Lemma \ref{injective}, $f|_{M_\theta\backslash C}$ is injective. By Proposition \ref{commute}, $f$ is surjective. By Proposition \ref{fiber}, $f|_{M_\theta\backslash C}:M_\theta\backslash C\to M_{\theta_0}\backslash T_0(\mathds{L})$ is also surjective. Since both $M_\theta\backslash C$ and $M_{\theta_0}\backslash T_0(\mathds{L})$ are smooth varieties, $f|_{M_\theta\backslash C}$ is an isomorphism by Zariski Main Theorem.
        
    \end{proof}

    \begin{thm}\label{final}
        $T:P^n_m\to M_\theta$ is an isomorphism and $f:M_\theta\to M_{\theta_0}$ is the blow down along $T_0(\mathds{L})$.
    \end{thm}
    \begin{proof}
        By Corollary \ref{open}, $f:M_{\theta}\to M_{\theta_0}\cong \mathds{P}^n$ is a proper birational morphism between smooth projective varieties. So by weak factorization theorem, $f$ can be factored into a sequence of blow ups with centers disjoint from $\mathds{P}^n\backslash\mathds{L}$. But since $C=\mathds{L}\times \mathds{P}^{n-m-1}$. $f$ is the blow up with center $\mathds{L}$.\\
        Now the fact that $T$ is an isomorphism follows immediately from universal property of blow ups.
    \end{proof}
    
\begin{rem}
    We would like to mention that \cite{Fei} constructed similar birational maps between quiver moduli in more general setting using categorical techniques.
\end{rem}

\begin{proof}[Proof of Theorem \ref{main}]
    We collection what we have proved. By Lemma \ref{slo}, the collection $\{\cO_{P^n_m},\cO_{P^n_m}(H-E),\cO_{P^n_m}(H)\}$ is strong exceptional. Taking $\theta=(-p,p-q,q)$ for $p,q\in \Z_{>0}$, $p<q$ as in this section, we can apply Lemma \ref{fine} to show $M_\theta$ is a fine moduli space. Finally Theorem \ref{final} shows $T$ is an isomorphism.
\end{proof}

\section{Proof of Theorem \ref{second}}
Let $\theta'=(-p,0,p)$.
\begin{lem}\label{stable3}
    $R$ is $\theta'$-semistable if and only if 
    \begin{align*}
        r_{x_i}\neq 0
    \end{align*}
    for at least one $0\leq i\leq m$ OR
    \begin{align*}
        r_{x_i}r_e\neq 0
    \end{align*}
    for at least one $m+1\leq i\leq n$.
\end{lem}
\begin{proof}
    By Lemma \ref{subrep}, a representation $R$ of $Q$ is $\theta'$-unstalbe if and only if it has subrepresentation with dimension either $(1,0,0)$ or $(1,1,0)$, which in turn is equivalent to the conditions in the statement of the lemma.
\end{proof}
\begin{rem}
    We note $M_{\theta'}$ is a coarse moduli space with strictly semistable representations. For example $(x_1,\ldots,x_m,0,\ldots,0, e)$ with $x_1\neq 0$ has a subrepresentation with dimension vector $(0,1,0)$, which makes it not stable.
\end{rem}

\begin{proof}[Proof of Theorem \ref{second}]
    To show $M_{\theta_0}\cong \mathds{P}^n$, simply notice $B(\lambda \theta_0)$ is the $\mathbf{k}$ span of $$\{x_1^{a_1}\ldots x_m^{a_m}(x_{m+1}e)^{a_{m+1}}\ldots (x_ne)^{a_n}\}$$
    for $a_1+\ldots+a_n=p$. Thus
    \begin{align*}
        M_{\theta_0}\cong \Proj\bigoplus_{\lambda\geq0}B(\lambda \theta_0)
    \end{align*}
    corresponds to the $p$-upple embedding of $\mathds{P}^n$ having $x_1,\ldots,x_m,(x_{m+1}e),\ldots,(x_ne)$ as projective coordinates.\\
    To show $id:Rep(Q)\to Rep(Q)$ descents to a morphism, we need to check if $R$ is $\theta$-stable, then $R$ is $\theta_0$-semistable. If $R\subset U$, then $R$ is $\theta$-stable implies $r_{x_i}\neq 0$ for at least one $m+1\leq i\leq n$ by Lemma \ref{stable1}, which implies $r_er_{x_i}\neq 0$ for at least one $m+1\leq i\leq n$, thus $R$ is $\theta_0$-semistable by Lemma \ref{stable3}.
    
    If $r_e=0$, then $R$ is $\theta$-stable implies $r_{x_i}\neq 0$ for at least one $1\leq i\leq m$ by Lemma \ref{stable2}, which in turn implies $R$ is $\theta_0$-semistable by Lemma \ref{stable3}. Thus $id$ descends.
    
    To see the induces morphism $\pi':M_\theta\to M_{\theta_0}$ is the natural projection, we recall that $M_{\theta_0}\cong \mathds{P}^n$ has  $x_1,\ldots,x_m,(x_{m+1}e),\ldots,(x_ne)$ as projective coordinates. Then following the definition of $f$ and the proof of Corollary \ref{open}, we see $\pi'|_{M_\theta\backslash C}$ is an isomorphism.\\
    It remains to observe that for $[a_1,\ldots,a_m,b_{m+1}\ldots,b_n,0]\in C\cong\mathds{P}^m\times\mathds{P}^{n-m-1}$,
    \begin{align*}
        \pi'([a_1,\ldots,a_m,b_{m+1}\ldots,b_n,0])=[a_1,\ldots,a_m,0,\ldots,0]
    \end{align*}
    where the right hand side is written in the coordinates of $x_1,\ldots,x_m,(x_{m+1}e),\ldots,(x_ne)$. Thus $\pi'$ is the blow down morphism.
\end{proof}

\begin{rem}
    The reason for $id$ to descent to a contraction is the existence of strictly semistable representations in $M_{\theta_0}$. There are two kinds of these representations:
    \begin{itemize}
        \item $[a_1,\ldots,a_n,0]$ where at least one of $a_i\neq 0$ for $1\leq i\leq m$.
        \item $[b_1,\ldots,b_m,0,\ldots,0,e]$ where at least one of $b_i\neq 0$ for $1\leq i\leq m$.
    \end{itemize}
    Notice the intersection of these two types are representations $[c_1,\ldots,c_n,0,\ldots,0,0]$. Moreover, letting $(1,k,1)$ acts on the first kind and $(1,1/k,1)$ acts on the second, and $k\to \infty$ we see there is a representation of form $[c_1,\ldots,c_n,0,\ldots,0,0]$ lying in the orbit closure of any strictly semistable representation. \\
    In particular, for all representation in $[a_1,\ldots,a_m]\times\mathds{P}^{n-m-1}\subset C\subset M_\theta$, their orbit closures in $M_{\theta_0}$  contain $[a_1,\ldots,a_m,0,\ldots,0,0]$, thus they are represented by one single point in $M_{\theta_0}$.
\end{rem}

\section*{Acknowledgements}
The author would like to thank his advisor Valery Lunts for constant support and inspiring discussions. He would like to thank Alastair Craw, Shizhuo Zhang for many helpful discussions and suggestions.

\printbibliography[heading=bibintoc,title={Bibliography}]
\end{document}